\documentclass{amsart}
\usepackage{amsfonts}
\usepackage{amssymb}
\usepackage{graphicx}

\newtheorem{thm}{Theorem}
 
 \newtheorem{lem}[thm]{Lemma}
 \newtheorem{prop}[thm]{Proposition}
\theoremstyle{definition}

  \DeclareMathOperator{\ii}{i} 
  \DeclareMathOperator{\IM}{Im}
 \DeclareMathOperator{\supp}{supp} \DeclareMathOperator{\dist}{dist}
\DeclareMathOperator{\tr}{tr} \DeclareMathOperator{\RE}{Re}
\DeclareMathOperator{\diag}{diag}
\newcommand{\Real}{\mathbb{R}}
\newcommand{\Comp}{\mathbb{C}}
\newcommand{\Nat}{\mathbb{N}}

\newcommand{\eps}{\varepsilon}
\newcommand{\dts}{,\dots,}

\newcommand{\sbs}{\subset}

\newcommand{\set}[1]{\left\{#1\right\}}

\newcommand{\norm}[1]{\left\Vert#1\right\Vert}

\newcommand{\matp}[1]{\begin{bmatrix} #1 \end{bmatrix}}

\providecommand{\norm}[1]{\left\Vert#1\right\Vert}

\begin{document}

   \title[  $H$--selfadjont random matrices]
   {On a class of $H$--selfadjont random matrices with one eigenvalue of nonpositive type.}
   \author{ Micha\l{} Wojtylak}
   \address{
   Faculty of Mathematics and Computer Science\\
   Jagiellonian University\\
   \L ojasiewicza 6\\
   30-348 Krak\'ow\\
   Poland}
   \email{michal.wojtylak@gmail.com
}

\begin{abstract}
Large $H$--selfadjoint random matrices are considered. The matrix $H$ is assumed to have one negative eigenvalue, hence the matrix in question has precisely one  eigenvalue of nonpositive type. It is showed that this eigenvalue converges in probability to a deterministic limit. The weak limit of distribution of  the real eigenvalues is investigated as well.   
\end{abstract}
\thanks{The research was supported by Polish Ministry of Science and Higher Education with a Iuventus Plus grant. }

\maketitle

\section*{Introduction}

The main object of this survey are random matrices that are not symmetric, but are selfadjoint with respect to an indefinite inner product. Spectrum and numerical range of some classes of such  matrices  were considered recently in \cite{NCR1,NCR2}, although,   in the present paper we consider different instances.  The inefinite linear algebra motivation of the present research is the follwoing.  
Consider an invertible, hermitian--symmetric matrix $H\in\Comp^{n\times n}$. We say that $X\in\Comp^{n\times n}$ is $H$--selfadjoint if $X^*H=HX$. This is  the same as to say that $A$ is selfadjoint with respect to an inner product
$$
[x,y]_H:=y^*Hx,\quad x,y\in\Comp^n.
$$
Note that this inner product is not positive definite if $H$ has negative eigenvalues.
In the literature the space $\Comp^n$ with the  inner product $[\cdot,\cdot]_H$ is also called $\Pi_\kappa$--space (where $\kappa$ is the number of negative eigenvalues of $H$) or Pontryagin space, the infinite dimensional case is considered as well \cite{bognar, langerio}.
In the present paper the  case when 
\begin{equation}\label{H}
H =\left[ 
\begin{array}{cc}
-1 & 0  \\ 
0 & I_{N} 
\end{array}
\right],
\end{equation}
 is considered. It is easy to check that for such $H$  each $H$--selfadjoint matrix has the form
 \begin{equation}\label{A}
X=\left[
\begin{array}{cc}
a & -b^*\\ 
b & C  \\
\end{array}
\right],
\end{equation}
with $x\in\Real$, $b\in\Comp^N$ and a hermitian--symmetric matrix $C\in\Comp^{N\times N}$.  
Due to the famous theorem of Pontryagin \cite{pontr} the matrix $X$ has precisely one eigenvalue $\beta$, for which the corresponding eigenvector $x$ satisfies $[x,x]_H\leq 0$. The problem of tracking the nonpositive eigenvalue  was considered for example  in \cite{DHS3,SWW}. In those papers the setting was non--random and $X$ was in the family of one dimensional extensions of a fixed operator in an infinite dimensional $\Pi_1$--space.  
 The aim of the present work is to investigate the behavior of $\beta$ when  $X$ is a large random matrix. We show that the main method of \cite{DHS3,SWW} -- the use of Nevanlinna functions with one negative square -- can be adapted to the random setting as well.  
 
 A classical result of Wigner \cite{wigner} says that if the random variables $y_{ij}$, $0\leq i\leq j<+\infty$ are real, i.i.d with mean zero and  variance equal one, then the distribution of eigenvalues of a matrix 
  $$
  Y_N=\frac{1}{\sqrt{N}} [y_{ij}]_{ij=0}^N,
  $$
where $y_{ji}=y_{ij}$ for $j>i$, converge weakly in probability to the Wigner semicircle measure. Note that by multiplying
the first row of $Y_N$ by -1 we obtain a $H$--selfadjoint matrix $X_N$. A result  of a preliminary numerical experiment with gaussian $y_{ij}$  is plotted in Figure \ref{fig}. 
\begin{figure}
\includegraphics[width=200pt]{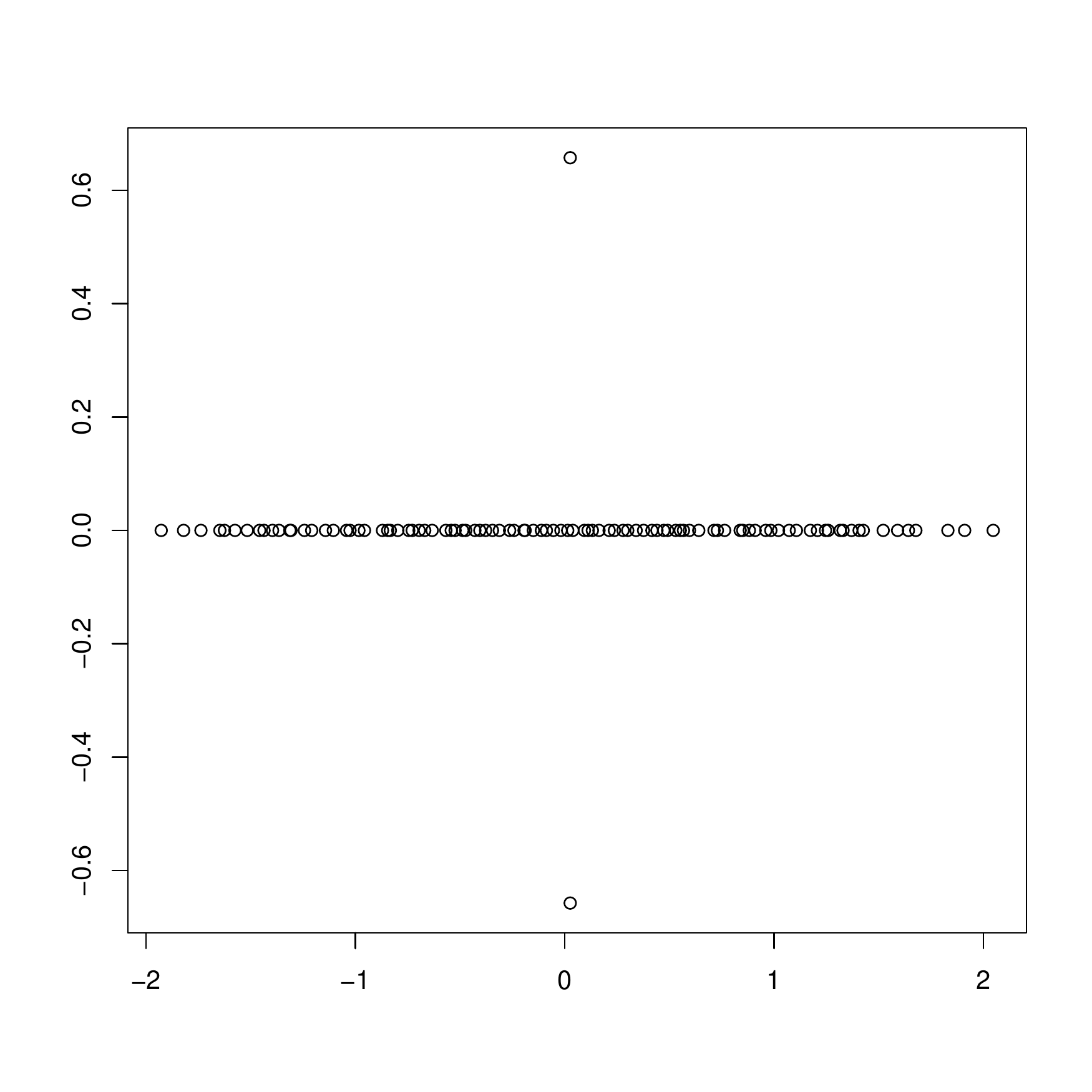}
\caption{Eigenvalues of a  random matrix $X_{100}$ computed with R \cite{R} }\label{fig} 
\end{figure}
Note  that the spectrum of $X_N$ is real, except two eigenvalues, lying symmetrically  with respect to the real line.  
Although we pay a special attention to the above case, we study the behavior of the eigenvalue of nonpositive type in a more general setting.
Namely, we assume that the random matrix $X_N$ in $\Comp^{N\times N}$ is of a form
$$
X=\left[
\begin{array}{cc}
a_N & -b_N^*\\ 
b_N & C_N  \\
\end{array},\right],
$$
with $a_N$, $b_N$ and  $C_N$ being independent. Furthermore,  the vector $b_N$ is  a column of a Wigner matrix and $a_N$ converges weakly to zero. The only assumption on $C_N$ is that the limit distribution of its eigenvalues converge weakly in probability, see (R0)--(R3) for details.  
In Theorem \ref{main} we prove  that under these assumptions the non--real eigenvlaues converge in probability  to a deterministic limit that can be computed knowing the limit distribution of eigenvalues of $C_N$. In the case when $C_N$ is a Wiegner matrix the nonreal eigenvalues converge to $\pm\ii\sqrt{2}/2$, cf. Theorem \ref{WiegnerTh}. Furthermore, under a technical assumption of continuity of the entries of $X_N$, we show in Theorem \ref{realev} that the limit distribution of the real eigenvalues of $X_N$  coincides  with the limit distribution of eigenvalues  for the matrices $C_N$. Again, in the case when $C_N$ is a Wiegner matrix we obtain a more precise result. Namely,  in Theorem \ref{WiegnerTh} we show that  the real  eigenvalues $\zeta_2^N\dts \zeta_N^N$ of $X_N$ and the eigenvalues  $\lambda_1^N\dts \lambda_N^N$ of $C_N$ satisfy the following inequalities: 
$$
\lambda_1^N < \zeta_2^N < \lambda_2^N < \cdots <\lambda_{N-1}^N < \zeta_{N}^N < \lambda_N^N.
$$
It shows that the  nonreal eigenvalue of $X_N$ plays an analogue role as the largest eigenvalue in one--dimensional, symmetric perturbations of Wiegner matrices. This fact  relates the present paper to  the current work on finite dimensional perturbations of random matrices, see \cite{BG1,BG2,BG3,cap1,capitaine,FP,knowles} and references therein.  Also note that $X_N$ is a product of a random  and deterministic matrix, such products  were already considered in the literature, see  e.g. \cite{vershynin}.

The author is indebted to Anna Szczepanek for preliminary numerical simulations and Maxim Derevyagin for his valuable comments. Special thanks to Anna Kula, Patryk Pagacz and Janusz Wysocza\'nski.

\section{Functions of class $\mathcal{N}_1$}\label{sN1}

The Nevanlinna functions with negative squares  play a similar role for the class of $H$--selfadjoint matrices as the class of ordinary Nevanlinna function plays for hermitian--symmetric matrices. This phenomenon has its roots in  in operator theory, we refer the reader to  \cite{DHS1, DHS3, KL71,KL77,KL81} and papers quoted therein for a precise description of a relation between   $\mathcal{N}_\kappa$--functions and selfadjoint operators  in Krein and Pontryagin spaces. We begin with a very general definition of the class $\mathcal{N}_\kappa$, but we immediately restrict ourselves to certain subclasses of those functions.

We say that $Q$ is \textit{a generalized Nevanlinna function of class $\mathcal N_{\kappa}$} \cite{KL71,  L} if it is meromorphic in the upper half--plane
$\Comp^+$ and the kernel 
$$
N(z,w)=\frac{Q(z)-\overline{Q(w)}}{z-\bar w}
$$
has precisely $\kappa$ negative squares,  that is for any finite sequence $z_1\dts z_k\in\Comp^+$  the hermitian--symmetric matrix
$$
[N(z_i,z_j)]_{ij=1}^k
$$
has not more then $\kappa$ nonpositive eigenvalues and for some choice of $z_1\dts z_k$ it has precisely $\kappa$ nonpositive eigenvalues. 
In the present paper we use this definition with $\kappa=0,1$. 

The class $\mathcal{N}_0$ is the class of ordinary Nevanlinna functions, i.e. the functions that are holomorphic in $\Comp^+$ with nonnegative imaginary part. By $M^+_b(\Real)$ we denote the set of positive, bounded Borel measures on $\Real$. For $\mu\in M^+_b(\Real)$  we define  the Stieltjes transform as
$$
\hat\mu(z)=\int_\Real \frac 1{t-z}dt ,\quad  z\in\Comp\setminus\supp\mu.
$$
 Clearly, $\hat\mu$ belongs to the class $\mathcal{N}_0$ and the values of $\hat\mu$ in the upper half--plane determine the measure uniquelly by the Stieltjes inversion formula. Although not every function of class $\mathcal{N}_0$ is a Stieltjes transform of a Borel measure (cf. \cite{donoghue}), this subclass of $\mathcal{N}_0$ functions will be sufficient for present reasonings. 
Also, we will be interested in a special subclass of $\mathcal{N}_1$ functions, namely in the functions of the form \eqref{Qform} below.  We refer the reader to the literature \cite{DHS1,DLLSh}  for representation theorems for $\mathcal{N}_\kappa$ functions.
\begin{prop}\label{N1}
 If  $\mu\in M_b^+(\Real)$, $a\in\Real$ then 
\begin{equation}\label{Qform}
Q(z)=\hat\mu(z) + a - z
\end{equation}
is a holomorphic function in $\Comp^+$ and belongs to the class $\mathcal{N}_1$. Furthermore, there exists precisely one $z_0\in\Comp$ such that either $z_0\in\Comp^+$ and
\begin{equation}\label{x_}
Q(z_0)=0,
\end{equation}
or $z_0\in\Real$ and
\begin{equation}\label{x_0}
\lim_{z \hat\to z_0}\frac{Q(z)}{z-z_0}\in (-\infty,0].
\end{equation}
\end{prop}
The symbol $\hat\to$ above denotes the non-tangential limit:
$$
z\in\Comp^+,\quad z\to z_0,\quad \pi/2-\theta \leq \arg(z-z_0)\leq \pi/2+\theta,
$$
with some $\theta\in(0,\pi/2)$. We call $z_0\in\Comp^+\cup\Real$  the \textit{generalized zero of
nonpositive type} (\textit{GZNT}) of $Q(z)$. 
The first part of the Proposition can be found e.g. in \cite{KL77}, while for the proof of the 'Furthermore' part in the general context\footnote{For arbitrary $\mathcal{N}_1$ function $z_0=\infty$ can be also the GZNT, in that case $\lim_{z \hat\to \infty} zQ(z) \in [0,\infty)$. However, this is clearly not possible for $Q$ of the form \eqref{Qform}.}
 we refer the reader to  \cite[Theorem 3.1, Theorem 3.1']{L}.
%
 In view of the above proposition we can define a function 
$$
G:M_b^+(\Real)\times\Real\to \Comp^+
$$
by saying that $G(\mu,a)$ is  the GZNT of the function $\hat\mu(z) + a - z$.
 The following proposition plays a crucial role in our arguments.

\begin{prop}\label{cont}
 The function $G$ is jointly continuous with respect to the weak topology on $M_b^+(\Real)$ and the standard topology on $\Real$.
\end{prop}

\begin{proof}
Assume that $(\mu_n)_n\sbs M^+_b(\Real)$ converges weakly to $\mu\in M^+_b(\Real)$ and $a_n\in\Real$ converges to $a\in\Real$ with $n\to\infty$. Take a compact $K$ in the open upper half--plane, with nonempty interior. Then $\hat\mu_n$ converges uniformly to $\hat\mu$ on the set $K$. Indeed, if $r=\sup_{t\in\Real,\ z\in K}1/|t-z|$ then
 $$
\sup_{z\in K}  |\hat\mu_n(z) - \hat\mu_0 (z)| \leq r |\mu_n - \mu_0|(\Real),
$$ 
the latter clearly converging to zero with $n\to\infty$. In consequence, $\hat\mu_n(z) + a_n - z$ converges to $\hat\mu(z) + a - z$ uniformly on $K$ with $n\to\infty$. By \cite{LaLuMa} the GZNT  of $\hat\mu_n(z) + a_n - z$ converges to the GZNT of $\hat\mu(z) + a - z$, which finishes the proof. 
\end{proof}

\section{$H$--selfadjoint matrices}
In this section we  review basic properties of selfadjoint matrices in indefinite inner product spaces introducing the concept of a canonical form and showing its relation with $\mathcal{N}_1$--functions.
Let $H\in\Comp^{(n+1)\times( n+1)}$ ($n\in\Nat\setminus\set0$) be an invertible, Hermitian--symmetric matrix. 
We say that $X\in\Comp^{(n+1)\times( n+1)}$ is \textit{$H$--selfadjoint} if $X^*H=HX$.
Our main interest will lie in the matrix
\begin{equation}\label{H}
H =\left[ 
\begin{array}{cc}
-1 & 0  \\ 
0 & I_{n} 
\end{array}
\right],
\end{equation}
where $I_n$ denotes the identity matrix of size $n\times n$. As it was already mentioned, each $H$--selfadjoint matrix has the form
 \begin{equation}\label{A}
X =\left[ 
\begin{array}{cc}
a & -b^*\\ 
b & C  \\
\end{array}
\right],
\end{equation}
with $a\in\Real$, $b\in\Comp^n$ and  hermitian--symmetric $C\in\Comp^{n\times n}$.
 Due to \cite{GLR}  there exists an invertible matrix $S$ and a pair of matrices $H',S'\in\Comp^{(n+1)\times(n+1)}$ such that  $X=S^{-1}X'S$ $H=S^* H' S$ and $X',H'$ are of one of the following forms:
 \begin{itemize}
 \item[Case 1.]  
$$
X'=\matp{ \beta & 0\\ 0 & \bar \beta} \oplus \diag (\zeta_2\dts \zeta_{n}),\quad H'=  \matp{ 0 & 1\\ 1 & 0} \oplus I_{n-1},
$$
with $\beta\in\Comp^+$, $\zeta_2\dts\zeta_{n}\in\Real$.

\item[Case 2.]
$$
X'= [\beta]\oplus\diag ( \zeta_1\dts \zeta_{n}),\quad H'= [-1] \oplus I_n,
$$
with $\beta\in\Real$, $\zeta_1\dts\zeta_{n}\in\Real$.

\item[Case 3.]
$$
X'=\matp{ \beta & 1\\ 0 &  \beta} \oplus \diag (\zeta_2\dts \zeta_{n}), \quad H'=\gamma  \matp{ 0 & 1\\ 1 & 0} \oplus I_{n-1},
$$
with $\beta\in\Real$, $\zeta_2\dts\zeta_{n}\in\Real$, $\gamma\in\set{-1,1}$.

\item[Case 4.] 
$$
X'=\matp{ \beta & 1 & 0 \\ 0 &  \beta & 1\\ 0 & 0 & \beta} \oplus \diag (\zeta_3\dts \zeta_{n}), \quad H'= \matp{ 0 &  0 & 1\\ 0 & 1 & 0\\ 1 & 0 & 0} \oplus I_{n-2},
$$
with $\beta\in\Real$, $\zeta_3\dts\zeta_{n}\in\Real$.

\end{itemize}
 It is easy to verify that in each case $X'$ is $H'$-symmetric. The pair $(X',H')$ is called \textit{the canonical form} of $(X,H)$. We refer the reader to \cite{GLR} for the proof and for canonical forms for general H--symmetric matrices and to \cite{bognar,langerio} for the infinite--dimensional counterpart of the theory.  At this point is enough to mention that the canonical form is uniquely determined (up to permutations of the numbers $\zeta_i$) for each pair $(X,H)$, where $X$ is $H$--selfadjoint. 
Note that in each of the cases $\beta$ is an eigenvalue of $X$ and there exists a corresponding eigenvector $x\in\Comp^{n+1}$ satisfying  $[x,x]_H\leq0$, furthermore, $\beta$ is the only eigenvalue in $\Comp^+\cup\Real$ having this property. Therefore, we will call $\beta$  the \textit{eigenvalue of nonpositive type of $X$}.  
 
Observe that the function 
\begin{equation}\label{QX}
Q(z)=a-z+b^*(C-z)^{-1}b
\end{equation}
is an $\mathcal{N}_1$--function. Indeed, if $D=UCU^*=\diag(\lambda_1\dts \lambda_n)$ is a  diagonalization of the hermitian--symmetric matrix $C$ and $d=Ub$ then
$$
Q(z)=a-z +\sum_{j=1}^n \frac{|d_j|^2}{\lambda_j-z}=a-z+\hat\mu(z),\quad\text{where}\quad  \mu=\sum_{j=1}^n|d_j|^2 \delta_{\lambda_j},
$$  
and we may apply Proposition \ref{N1}.
The following  lemma is a standard in the indefinite linear algebra theory. We present the proof for the reader's convenience.

\begin{lem}\label{ZZ}
Let $X$ and $Q$ be defined  by \eqref{A} and \eqref{QX}, respectively. A point $\beta\in\Comp^+\cup\Real$ is the eigenvalue of nonpositive type of $X$ if and only if it is the GZNT of $Q(z)$. Furthermore, the algebraic multiplicity of $\beta$ as an eigenvalue of $X$ equals the order of $\beta$ as a zero of $Q(z)$.
\end{lem}

\begin{proof}
First note that, due to the Shur complement formula\footnote{It is well known \cite{KL77} that $-1/Q$ belongs to $\mathcal{N}_1$ provided that $Q$ belongs to $\mathcal{N}_1$, however, this information is not essential for the proof.},
$$
-\frac{1}{Q(z)}={e^*H(X-z)^{-1}e},
$$
where $e$ denotes the first vector of the canonical basis of $\Comp^{n+1}$. Let  $(X',H')$ be the canonical form of $(X,H)$ and let $S$ be the appropriate transformation. Consequently,
\begin{equation}\label{-1/Q}
-\frac{1}{Q(z)}= {(Se_1)^* H' (X'-z)^{-1} Se_1}.
\end{equation}
Below we evaluate this expression in each of the Cases 1--4. Let $f=[f_0\dts f_n]^\top=Se$. Note that 
 \begin{equation}\label{ef}
f_1^* H'  f_1= e_1^*He_1=-1,
\end{equation}
independently on the Case. 

Case 1. Observe that $f_0\bar f_1\neq 0$, otherwise $f^*H'f \geq 0$, which contradicts \eqref{ef}. Due to \eqref{-1/Q} one has
$$
-\frac{1}{Q(z)}=  \frac{f_0\bar f_1}{\beta-z}+\frac{f_1 \bar f_0}{\bar\beta-z} +\sum_{j=2}^{n} \frac{ |f_{j}|^2}{\zeta_j-z}.
$$
Hence, $\beta\in\Comp^+$ is a simple pole of $-1/Q$ and consequently it is the GZNT of $Q$ and a simple zero of $Q$. 

Case 2. Observe that $|f_0|^2>\sum_{j=1}^n|f_j|^2$, otherwise  $f^*H'f \geq 0$, which contradicts \eqref{ef}. Due to \eqref{-1/Q} one has
$$
-\frac{1}{Q(z)}=  \frac{-|f_0|^2}{\beta-z}+\sum_{j=1}^{n} \frac{ |f_{j}|^2}{\zeta_j-z}.
$$
 Hence, the residue of  $-1/Q$ in $\beta$ is less then zero. Consequently $Q(\beta)=0$, $Q'(\beta)<0$ and $\beta$ is the GZNT of $Q$. 

Case 3. Observe that  $|f_1|^2>0$,  otherwise  $f^*H'f \geq 0$, which contradicts  \eqref{ef}. Due to \eqref{-1/Q} one has
$$
-\frac{1}{Q(z)}=  \frac{2\gamma\RE f_0\bar f_1}{\beta-z}+  \frac{-\gamma|f_1|^2}{(\beta-z)^2}+ \sum_{j=2}^{n} \frac{ |f_{j}|^2}{\zeta_j-z}.
$$
 Hence, $\beta$ is pole of $-1/Q$ of order 2. Consequently, $Q(\beta)=Q'(\beta)=0$, $Q''(\beta)\neq 0$ and $\beta$ is the GZNT. 

Case 4. Observe that $|f_2|^2>0$,  otherwise $f^*H'f\geq 0$, which contradicts \eqref{ef}. Due to \eqref{-1/Q} one has
$$
-\frac{1}{Q(z)}=  \frac{2\RE f_0\bar f_2+|f_1|^2}{\beta-z}+  \frac{-2\RE f_1\bar f_2}{(\beta-z)^2}+\frac{|f_2|^2}{(\beta-z)^3}+ \sum_{j=3}^{n} \frac{ |f_{j}|^2}{\zeta_j-z},
$$
  Hence, $\beta$ is pole of $-1/Q$ of order 3. Consequently, $Q(\beta)=Q'(\beta)=Q''(\beta)=0$, $Q'''(\beta)\neq0$ and $\beta$ is the GZNT of $Q$.
\end{proof}

\section{Random $H$--selfadjoint matrices}

By $X_N$, $H_N$ we understand the following pair of a random and deterministic matrix in $\Comp^{(N+1)\times(N+1)}$
\begin{equation}\label{XH}
X_N=\left[ 
\begin{array}{cc}
a_N & -b_N^*\\
b_N & C_N \\
\end{array}
\right], \quad 
H_N=\matp{-1 & 0 \\ 0 & I_{N}},
\end{equation}
where
$a_N$ is a real--valued random variable, $b_N$ is a random vector in $\Comp^N$, and  $C_N$ is a hermitian--symmetric  random matrix in $\Comp^{N\times N}$. Note that $X_N$ is $H_N$--symmetric. By $\lambda_1^N\leq\cdots \leq\lambda_N^N$ we denote the eigenvalues of $C_N $ and by  $\nu_N$ we denote the random measure on $\Real$ 
 $$
 \nu_N=\frac 1N \sum_{j=1}^N \delta_{\lambda_j^N}.
 $$
 Recall that
 \begin{equation}\label{RStj}
 \hat \nu_N(z) = \frac { \tr (C_N-z)^{-1}}N.
 \end{equation}
 The assumptions on $X_N$ are as follows:
\begin{itemize}
\item[(R0)] The random variable $a_N$ is independent on the entries of the vector $b_N$ and on the entries of the matrix $C_N$ for each $N>0$, futhermore $a_N$   converges with $N\to \infty$ to zero  in probability.
\item[(R1)] The random vector $b_N$ is of the form
$$
b_N:= \frac1{\sqrt{N}}[x_{j0}]_{j=1\dts N},
$$
where $[x_{j0}]_{j>0 }$ are i.i.d. random variables, independent on the entries of $C_N$ for $N>0$, of zero mean  with $E|x_{j0}|^2={s}^2$  for $ j>0 $.
\item[(R2)]  The random measure $\nu_N$  converges with $N\to \infty$  to some non--random measure $\mu_0$ weakly in probability  
\end{itemize}

 All the results below hold also in the case when all variables $x_{j0}$ ($j>0$) are real, in this situation $b_N^*$ is just the transpose of $b_N$.  The entries of $C_N$ might be as well real or complex. In Section \ref{sW} we will consider two instances of the matrix $C_N$: a Wiegner matrix  and a diagonal matrix.
 In the case when $C_N$ is a Wiegner matrix the proposition below is a consequence of the isotropic semicircle law \cite{Erdos,knowles}.  We present below a simple proof of the general case, based on the ideas in \cite{MP69}.

\begin{prop}\label{BQ}
Assume that {\rm (R1)} and {\rm (R2)} are satisfied. Then for each $z\in\Comp^+$ 
$$
b_N^*(C_N -z)^{-1} b_N \to {s}^2\ \hat\mu_0(z) \quad (N\to\infty)
$$
 in probability.
\end{prop}

By $\norm y$ we denote the euclidean norm of $y\in\Comp^n$.

\begin{proof}
 In the light of Chebyshev's inequality, \eqref{RStj} and assumption (R2) it is enough  to show that 
\begin{equation}\label{rr}
 \lim_{N\to\infty}E\left|b_N^*(C_N-z)^{-1}b_N - {s}^2\frac{\tr (C_N-z)^{-1}}N\right|^2 =0.
\end{equation}
 Observe that
 \begin{equation}\label{summand1}
E \left|b_N^*(C_N-z)^{-1} b_N -{s}^2\frac{\tr (C_N-z)^{-1}}N\right|^2=
 \end{equation}
$$\left(E\left|b_N^*(C_N-z)^{-1} b_N\right|^2 -  {s}^4E\left|\frac{\tr (C_N-z)^{-1}}N\right|^2 \right)-
$$
\begin{equation}\label{summand2}
  2\RE E\left({s}^2\overline{\frac{\tr (C_N-z)^{-1}}N}\left(b_N^*(C_N-z)^{-1} b_N-{s}^2 \frac{\tr (C_N-z)^{-1}}N\right)\right). 
\end{equation}
First we prove  that the  summand  \eqref{summand2} equals zero. Indeed,  conditioning on  the $\sigma$--algebra generated by the entries of the matrix $C_N$ and setting
$$
[c_{ij}]_{ij=1}^N=(C_N-z)^{-1}
$$ 
one obtains
$$
   E\left({s}^2\overline{\frac{\tr (C_N-z)^{-1}}N}\left(b_N^*(C_N-z)^{-1} b_N- {s}^2\frac{\tr (C_N-z)^{-1}}N\right)\right) =
 $$ $$
 E\left( {s}^2 \sum_{i=1}^N \frac{\overline{ c_{ii}}}N \left( \sum_{jk=1}^N c_{jk}\frac{x_{0j}\overline{x_{0k}}}{N} - {s}^2 \sum_{j=1}^N \frac{c_{jj}}N\right) \right)=
 $$ $$
 E\left( {s}^2 \sum_{i=1}^N \frac{\overline{ c_{ii}}}N \left( \sum_{j=1}^N c_{jj}\frac{s^2}{N} - {s}^2 \sum_{j=1}^N \frac{c_{jj}}N\right) \right)=0.
 $$

Next, observe that
$$
E|b_N^*(C_N-z)^{-1} b_N|^2 =  E \sum_{ijkl=1}^N  c_{ij} \overline{c_{kl}}\frac{ x_{0i} \overline{ x_{0j}} x_{0k} \overline{x_{0l}}}{N^2} = 
$$ $$
 {s}^4 \sum_{ij=1}^N \frac{ E( c_{ii} \overline{ c_{jj}} )}{N^2} + {s}^4  \sum_{ij=1}^N \frac{E(c_{ij}\overline{c_{ij}})}{N^2}=
{s}^{4}E\left|\frac{ \tr (C_N-z)^{-1}}{N}\right|^2 +{s}^{4}E \sum_{ij=1}^N \frac{c_{ij}\overline{c_{ij}}}{N^2}.
$$
This allows us to  estimate  \eqref{summand1} by
$$ 
  {s}^4 \left| E \sum_{ij=1}^N \frac{c_{ij}\overline{c_{ij}}}{N^2} \right| \leq
  {s}^4 E\frac{\norm{(C_N-z)^{-1}}^2}{N}=\frac{{s}^4 \dist(z,\sigma( C_N))^{-2}}N\leq \frac{{s}^4 }{\IM z^2N},
$$
which  finishes the proof of \eqref{rr}.


\end{proof}

Let $U_N$ be a unitary matrix, such that $U_NC_NU^*_N$ is diagonal and let $d_N=[d^N_1\dts d^N_N]^\top=U_Nb_N$. 
Denote by $\mu_N$ the measure defined by
$$
\mu_N = \sum_{j=1}^N  | d_{j}^N |^2 \delta_{\lambda^N_j},
$$
and observe that $\hat\mu_N(z)=b_N^*(C_N-z)^{-1}b_N$. 

\begin{prop}\label{conv}
Assume that {\rm(R1)} and {\rm(R2)} are satisfied. Then the sequence of random measures $\mu_N$ converge weakly with $N\to\infty$ to $\mu_0$ in probability.
\end{prop}
\begin{proof}
First note that almost surely $\mu_N(\Real) \to {s}^2\mu_0(\Real)$ with $N\to \infty$. Indeed,
$$
\mu_N(\Real)= \sum_{j=1}^N |d_j^N|^2 = \norm{d_N}^2 = \norm{b_N}^2=\frac1N\sum_{j=1}^N |x_{0j}|^2,
$$
which converges almost surely to ${s}^2$ by the strong law of large numbers. Furthermore,
Proposition \ref{BQ} shows that $\hat\mu_N(z)$ converges in probability to  $\hat\mu_0(z)$ for every $z\in\Comp^+$.
Repeating the proof of Theorem 2.4.4 of \cite{AGZ}  we get the weak convegence of $\mu_N$ in probability. 
\end{proof}

\section{Main results}

\begin{thm}\label{main}
 If {\rm(R0) -- (R2)} are satisfied then the eigenvalue of nonpositive type $\beta_N$ of $X_N$ converges in probability to the GZNT $\beta_0$ of the $\mathcal{N}_1$--function 
 $$
 Q_0(z)=-z+{s}^2\hat\mu_0(z).
 $$
\end{thm}

\begin{proof}
 Consider a sequence of  $\mathcal{N}_1$--functions 
\begin{equation}\label{QN}
Q_N(z)=a_N-z+\hat\mu_N(z). 
\end{equation}
Recall that each of those functions has precisely one GZNT which, by definition of $\mu_N$ and Lemma \ref{ZZ}, is the 
eigenvalue  of nonpositive type $\beta_N$ of $X_N$. Recall that $a_N$ converges to zero in probability by (R0) and $\mu_N$ converges to $\mu_0$ in probability by Proposition \ref{conv}. Let $d$ be any metric that metrizises the topology of weak convergence on $M_b^+(\Real)$. Since  $\beta_N$ is a continuous function of $\mu_N$ and $a_N$ (Proposition \ref{cont}), for each $\eps>0$ one can find $\delta >0$ such that for each $N>0$ the event $\set{|a_N|<\delta,\  d(\mu_N,\mu_0)<\delta}$ is contained in 
$\set{ |\beta_N - \beta_0|<\eps}$. 
Using the assumed in (R0) independence of $\mu_N$ and $a_N$  one obtains
$$
P(|\beta_0-\beta_N|\geq\eps) \leq P(|a_N|\geq\delta) \cdot P( d(\mu_N,\mu_0)\geq\delta ).
$$
Hence, 
$\beta_N$ converges to $\beta_0$ in probability.
  \end{proof}

As it was explained in Section \ref{N1}, each matrix $X_N$ has, besides the eigenvalue $\beta_N$ of nonpositive type, a set of real eigenvaules $\zeta_{k_N}^N\dts \zeta_N^N$, where $k_N=1$ in Case 1 and 3, $k_N=2$ in Case 2 and $k_N=3$ in Case 3. 
By $\tau_N$ we denote the empirical measure connected with these eigenvalues:
$$
 \tau_N=\frac1N\sum_{j=k_N} ^N \delta_{\zeta_j}^N.
$$ 

\begin{thm}\label{realev}
If {\rm(R0)--(R2)} are saisfied and the random variables $\set{x_{0j}: j>0}$ are  continuous, then the measure $\tau_N$ converges   weakly in probability to $\mu_0$.  
\end{thm}

\begin{proof}
We use the notations $U_N,d_N$ and $\mu_N$ from the previous section, let also $Q_N$ be given by \eqref{QN}. 
  Note that the set $\set{y\in\Comp^N:(U_N y)_{j}=0}$ is of Lebesgue measure zero. Hence, with probability one $d_j^N\neq0$ for $j=1\dts N$, $N>0$. 
Therefore,  $\hat\mu_N(z)$ is a rational function almost surely with poles of order one in $\lambda_1^N\dts \lambda_N^N$. Furthermore, 
$$
Q_N(z) = \frac{(a_N-z)\ \prod_{j=1}^N(\lambda^N_j-z)+\sum_{i=1}^N|d^N_i|^2\prod_{j\neq i} (\lambda^N_j-z)    }{\prod_{j=1}^N(\lambda^N_j-z)}.
$$
In consequence, $Q_N$ has exactly $N+1$ zeros counting multiplicities, all of them different from $\lambda_1^N\dts\lambda_N^N$. 
Due to the Schur complement argument, each of those zeros is an eigenvalue of the matrix $X_N\in\Comp^{(N+1)\times(N+1)}$. 
Furthermore, due to Lemma \ref{ZZ} the algebraic multiplicity  of $\beta_N$ as eigenvalue of $X_N$ equals the order of $\beta_N$ as a zero of $Q_N$. In consequence, the spectrum of $X_N$ coincides with the zeros of  $Q_N$ and $\beta_N$ is the only zero of order possibly greater then one\footnote{In other words: $e$ is almost surely a cyclic vector of $X_N$.}.

On the other hand, the function $\hat\mu_N$ is increasing on the real line with simple poles in $\lambda_1^N\dts \lambda_N^N$. Hence,  in each of the intervals $(\lambda_j^N,\lambda_{j+1}^N)$ $(j=1\dts N-1)$ there is an odd number of zeros of $Q_N$, counting multiplicities.  Consequently,  in each of the intervals $(\lambda_j^N,\lambda_{j+1}^N)$ $(j=1\dts N-1)$ there is precisely one zero of $Q_N$, except possibly one interval that  contains three zeros of $Q_N$.  Out of these three zeros of $Q_N$ either one or two of them belong to the set 
$\set{\zeta_{k_N}^N\dts \zeta_{N}^N}$, accordingly to the canonical form of $X_N$. 
Hence,  in each of the intervals $(\lambda_j^N,\lambda_{j+1}^N)$ $(j=1\dts N-1)$ there is precisely one of the eigenvalues $\zeta_{k_N}^N\dts \zeta_{N}^N$, except possibly one interval that  contains two of the eigenvalues $\zeta_{k_N}^N\dts \zeta_{N}^N$. Consequently, the weak limit 
of $\tau_N$ in probability equals the weak limit of $\nu_N$. 

\end{proof}

\section{Two instances}\label{sW}
In the present section we consider two instances of  $C_N$: the  Wigner matrix and the  diagonal matrix. These both cases appear naturally as applications of main results. We refer the reader to \cite{P72} for a scheme joining both examples.

Consider an $H$--selfadjoint real Wigner matrix 
\begin{equation}\label{XW}
X_N:= \frac1{\sqrt{N}}\ H_N  [x_{ij}]_{ij=0}^N,
\end{equation}
with $x_{ij}$ real,  $x_{ij}= x_{ji}$ $(0\leq i < j <\infty)$,  i.i.d., of zero mean and variance equal to ${s}^2$,  and let  $H_N$ be defined as in \eqref{XH}.  Clearly $X_N$ is $H_N$--selfadjoint and satisfies (R0)--(R2) with $\mu_0$ equal to the Wiegner semicircle measure $\sigma$. The Stieltjes transform of the $\sigma$ equals 
$$
\hat\sigma(z) = \frac{-z+\sqrt{z^2 - 4{s}^2}}{2{s}^2}.
$$
 It is easy to check that $\beta_0=\frac{ \sqrt{2}}2 {s}\ii$ is a zero of $Q_0(z)=-z+{s}^2\hat\sigma(z)$.  Hence, $\beta_0$ is the GZNT of $Q_0$ and we have proved the first part of the theorem below.

\begin{thm}\label{WiegnerTh} Let  $X_N$ be defined by \eqref{XW}. Then
\begin{itemize}
\item [(i)]  $\beta_N$ converges in probability to 
$
\beta_0=\frac{ \sqrt{2}}2 {s}\ii;
$
\item[(ii)] if, additionally, the random variables $b_{ij}$ $(0\leq i<j<\infty)$ are continuous, then 
the probability of an event that there are precisely $N-1$ real eigenvalues $\zeta^N_2<\cdots <\zeta^N_N$ of $X_N$ and the 
inequalities
\begin{equation}\label{lambdazeta}
\lambda_1^N < \zeta_2^N < \lambda_2^N < \cdots <\lambda_{N-1}^N < \zeta_{N}^N < \lambda_N^N. 
\end{equation}
are satisfied, converges with $N$ to 1. 
\end{itemize}
\end{thm} 
\begin{proof}(ii) Assume that 
\begin{equation}\label{beta}
|\beta_N-\beta_0|\leq \frac{\sqrt2}4s.
\end{equation}
 Then the canonical form of $(X_N,H_N)$ is as in Case 1. 
In consequence, there are exactly $N-1$ real eigenvalues $\zeta^N_2\dts \zeta^N_N$ of $X_N$. 
 Let us recall now the arguments from proof of Theorem \ref{realev}. 
The function $\hat\mu_N$ is increasing on the real line with simple poles in $\lambda_1^N\dts \lambda_N^N$. In each of the intervals $(\lambda_j^N,\lambda_{j+1}^N)$ $(j=1\dts N-1)$ there at least one of the eigenvalues $\zeta_{2}^N\dts \zeta_{N}^N$. 
Consequently,  each of the intervals $(\lambda_j^N,\lambda_{j+1}^N)$ $(j=1\dts N-1)$ contains precisely one of the eigenvalues $\zeta_{2}^N\dts \zeta_{N}^N$. To finish the proof it is enough to note that by point (i) for every $\eps>0$ there exists $N_0>0$  such that for $N>N_0$ the probability of \eqref{beta} is greater then $1-\eps$.
\end{proof}

\begin{figure}
\includegraphics[width=200pt]{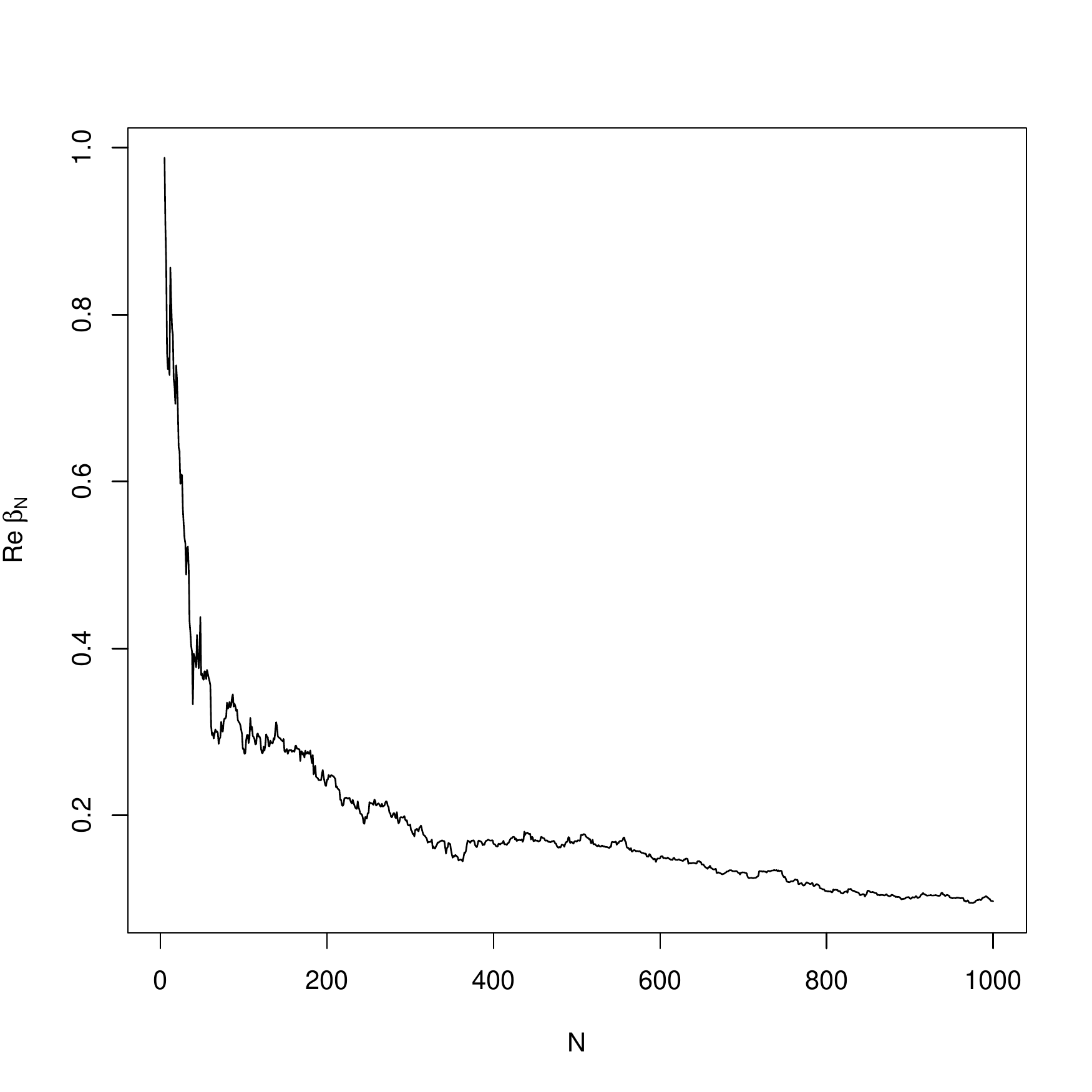}\\
\vskip -0.9cm
\includegraphics[width=200pt]{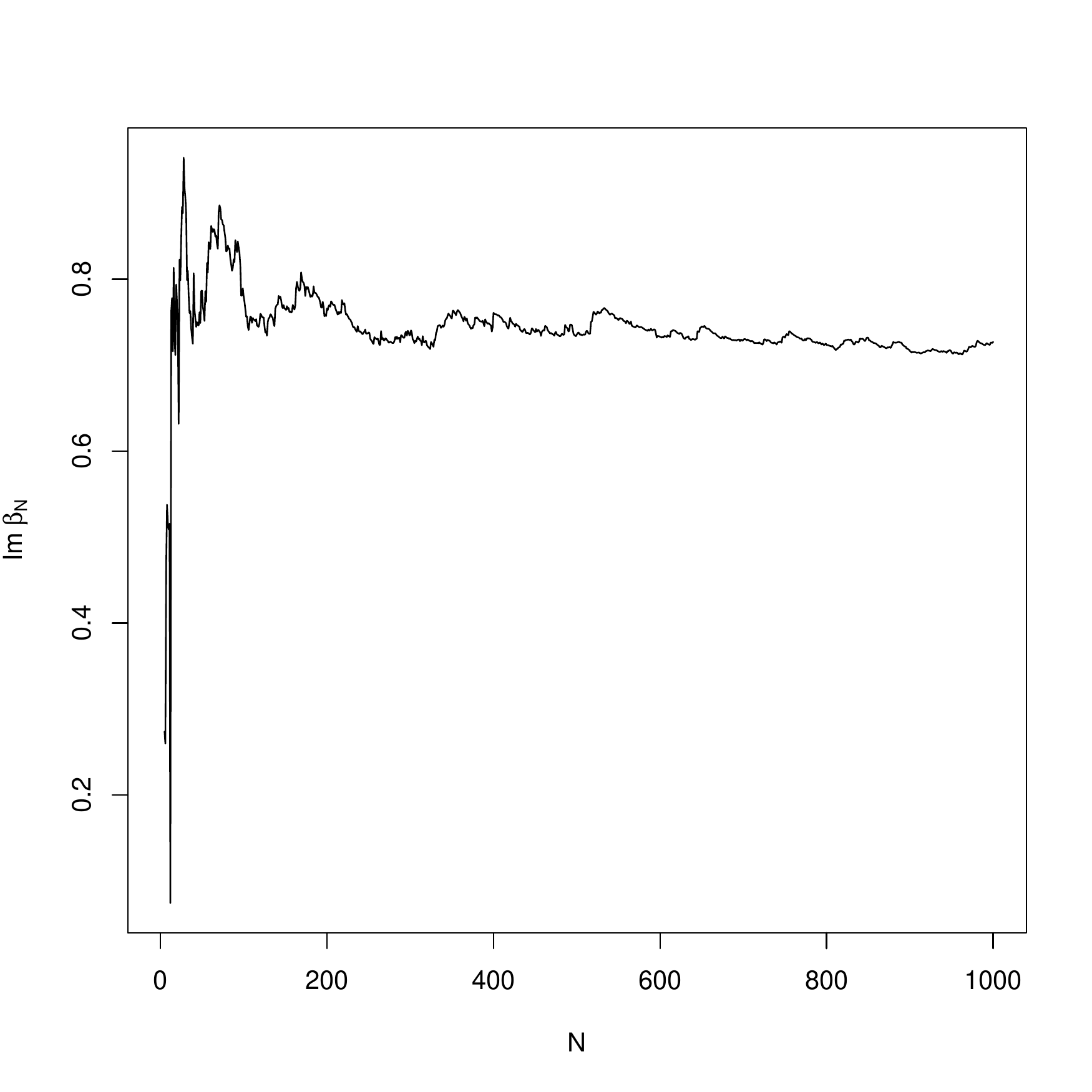}
\caption{The real and imaginary part of  $\beta_N$, with real, gaussian  entries of $X_N$ and $s^2=1$, computed with R \cite{R}.}\label{figreim}
\end{figure}

The numerical simulations of values of $\RE\beta_N$ and $\IM\beta_N$ can be found in Figure \ref{figreim}. 
Note that $\beta_0$ lies  in open upper half--plane and \eqref{x_} is satisfied. We provide now an example when $\beta_0\in\Real$ and show that each number in $[0,\infty)$ can be the limit in \eqref{x_0}. 
 Let $a_N=0$, ${x_{i0}}$ ($i=1,2,\dts$) be independent real variables of zero mean and variance ${s}^2$  and let 
 $C_N=\diag(c_1\dts c_N)$, where the random variables $\set{c_j:j=1,\dots}$ are i.i.d. and independent on  ${x_{i0}}$ ($i=1,2,\dts$). Furthermore, let the law of $c_j$ (which is simultaneously the limit measure $\mu_0$) be given by a density 
 $$
 \phi(t)=\begin{cases}\frac{ 3 t^2}2 &:  t\in [-1,1] \\ 0 &: t\in\Real\setminus[-1,1]\end{cases} .
 $$
An easy calculation shows that
$$
\lim_{z\hat\to 0}\frac{\hat\mu_0(z)}{z}= 3.
$$
Hence,  
$$
\lim_{z\hat\to 0}\frac{-z+{s}^2\hat\mu_0(z)}{z}=-1+ 3{s}^2
$$
and the function 
$$
Q_0(z)=-z+\hat\mu_0(z)=-z+\int_{\Real}\frac{\phi(t) }{t-z} dt
$$  
has a GZNT at $z=0$ if ${s}^2\leq1/3$. Note that $\beta_0=0$ lies in the support of $\mu_0$. The case ${s}^2=1/3$    is plotted in Figure \ref{F3}. Only the imaginary part is displayed, since the numerical computation of the real part of  $\beta_N$ might be not reliable in case $\beta_N\in\Real$. One may observe that the convergence of $\beta_N$ is worse in Figure  \ref{figreim}. Also, the canonical form of $(X_N,H_N)$ changes with $N$, contrary to the case when $H_NX_N$ is a Wigner matrix. In  the case ${s}^2<1/3$ in numerical simulations point $\beta_N$ is real for all $N$. 
\begin{figure}
\includegraphics[width=200pt]{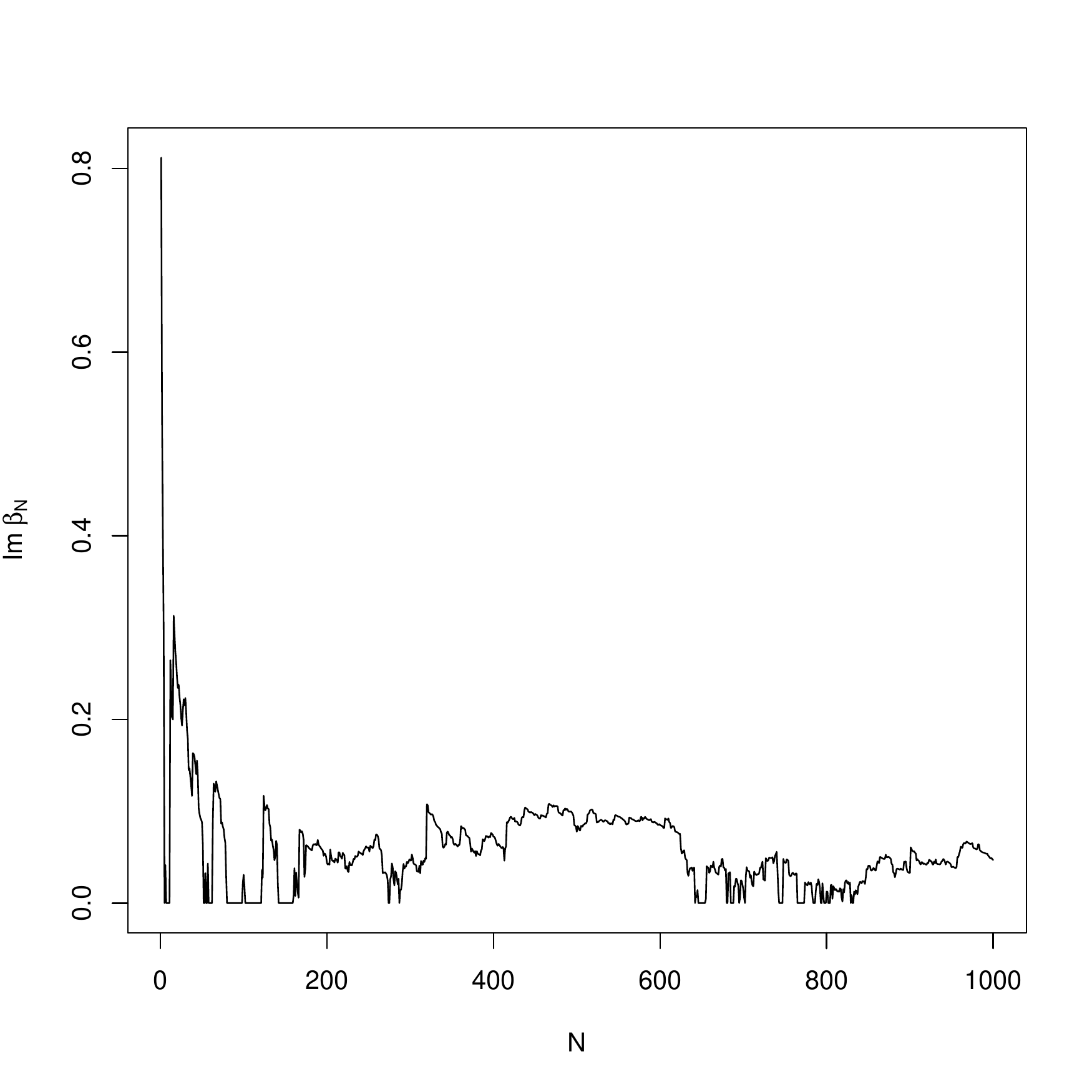}
\caption{The imaginary part of $\beta_N$.}\label{F3}
\end{figure}



\end{document}